\documentclass{amsart}
\usepackage{mathpazo}
\usepackage{amssymb}
\usepackage{newtxmath}
\newtheorem{theorem}{Theorem}[section]

\newtheorem*{Acknowledgment}{\textnormal{\textbf{Acknowledgment}}}
\theoremstyle{definition}
\newtheorem{definition}[theorem]{Definition}

\newtheorem{corollary}[theorem]{Corollary}
\newtheorem{proposition}[theorem]{Proposition}

\numberwithin{equation}{section}
\newcommand{\beqa}{\begin{eqnarray*}}
\newcommand{\eeqa}{\end{eqnarray*}}
\newcommand{\beqn}{\begin{eqnarray}}
\newcommand{\eeqn}{\end{eqnarray}}

\newcommand{\ov}{\overline}

\renewcommand{\a}{\alpha}

\newcommand{\la}{\lambda}

\newcounter{cnt1}
\newcounter{cnt2}
\newcounter{cnt3}
\newcommand{\blr}{\begin{list}{$($\roman{cnt1}$)$}
        {\usecounter{cnt1} \setlength{\topsep}{0pt}
                \setlength{\itemsep}{0pt}}}
\newcommand{\bla}{\begin{list}{$($\alph{cnt2}$)$}
        {\usecounter{cnt2} \setlength{\topsep}{0pt}
                \setlength{\itemsep}{0pt}}}
\newcommand{\bln}{\begin{list}{$($\arabic{cnt3}$)$}
        {\usecounter{cnt3} \setlength{\topsep}{0pt}
                \setlength{\itemsep}{0pt}}}
\newcommand{\el}{\end{list}}
\newtheorem{thm}{Theorem}

\newtheorem{Def}[thm]{Definition}

\newtheorem{rem}[thm]{Remark}
\newcommand{\Rem}{\begin{rem} \rm}
\newcommand{\bdfn}{\begin{Def} \rm}
\newcommand{\edfn}{\end{Def}}

\title{Nonrough norms in spaces with small diameter} 
\author[ S. Basu , S. Seal ]
	{Sudeshna Basu$^{1}$, Susmita Seal$ ^{2}$ }
\address{{$^{1}$}   Sudeshna Basu,
		Department of Mathematics and Statistics, 
				Loyola University, 
				Baltimore, MD 21210 and
				Department of Mathematics,
				George Washington University,
				Washington DC 20052 USA 
		}
	\email{sudeshnamelody@gmail.com}

\address {{$^{2}$} Susmita Seal, 
		Department of Mathematics,
		Ramakrishna Mission Vivekananda Education and Research Institute , 
		Belur Math,  Howrah 711202,
		West Bengal, India}
	\email{susmitaseal1996@gmail.com}	

\date{}
\begin{document}
\maketitle
\begin{abstract}
	In this work, we study the non rough norms  in Banach spaces with small diameter properties, namely the Ball Dentable Property ($BDP$), the Ball Huskable Property ($BHP$) and the Ball Small Combination of Slice Property ($BSCSP$).
 We introduce two more notions of non rough norms, namely the  weakly average non rough norms and average non rough norms in Banach spaces. We prove the duality between these three versions  of non rough norms and the   small diameter properties in Banach spaces. We also prove that  each of the three non rough norms is a three space property under certain assumptions.

\end{abstract}
\section{Introduction}
Throughout this work, we only consider {\it real} Banach space and
 $X^*$ stands for the dual of $X.$ We will denote the closed unit ball, the unit sphere and the closed ball of radius $r >0$ and center $x$ by $B_X$, $S_X$ and $B_X(x, r)$ respectively. 
  We refer to the monograph \cite{B1} for notions of convexity theory that we will be using here. 
 Let us begin with some definitions in geometry of Banach spaces that we need in our subsequent discussion.

\begin{definition} 
For a Banach space $X$ and bounded set $C\in X,$
\begin{enumerate}
\item a slice in $C$ is a set of the form 
$ S(C, x^*, \a) = \{x \in C : x^*(x) > \mbox{sup}~ x^*(C) - \a \}$
where  $x^* \in X^*$ and $\a > 0.$ One can analogously define $w^*$-slices in $X^*$ by choosing the determining functional from the predual space.
A convex combination of slices $S_i,$ $i =1 \ldots n$  is a set  $S =
\sum_{i=1}^{n} \la_i S_i$ where $\lambda_i>0 $ \ for all $i= 1,2\ldots n$ and $\sum_{i=1}^{n}\lambda_i =1 .$ 
\item a point $x\in C$ is called an exposed point of C if there exists $x^*\in X^*$ such that $x^*(x)>x^*(y)$ if $y\in C\setminus \{x\}$ and $x^*$ is said to expose $x.$
\item a point $x\in C$ is called a strongly exposed point of $C$ if there exists $x^*\in X^*$ such that $x^*$ exposes $x$ and $\{S(C,x^*,\alpha):\alpha>0\}$ is a local base at $x$ in $C$ for the norm topology.
\end{enumerate}
\end{definition}

\begin{definition}
\cite{BS}
 A Banach space $X$ has 
\begin{enumerate}
\item   Ball Dentable Property ($BDP$) if $B_X$ has  
   slices of arbitrarily small diameter. 
\item   Ball Huskable Property ($BHP$) if $B_X$ 
	has nonempty relatively weakly open subsets of arbitrarily small diameter.
\item   Ball Small Combination of Slice Property ($BSCSP$) 
	if  $B_X$ has convex combination of slices of arbitrarily 
	small diameter.
	\end{enumerate}
 \end{definition}
One can analogously define $w^*$-$BDP$, $w^{*}$-$BHP$ and $w^{*}$-$BSCSP$ in a dual space where slices and weakly open sets are replaced by $w^*$-slices and $w^*$-open sets  respectively.
 We have the following implications. 
$$ BDP \Longrightarrow \quad BHP \Longrightarrow \quad  BSCSP$$ $\quad \quad \quad \quad \quad \quad \quad \quad\quad\quad\quad\quad\quad\quad\quad\quad \Big \Uparrow \quad \quad\quad\quad\quad \Big \Uparrow \quad \quad\quad\quad\quad \Big \Uparrow$  $$ w^*BDP \Longrightarrow  w^*BHP \Longrightarrow  w^*BSCSP$$
In general, none of the reverse implications hold good. For details, see \cite{BS}. In fact, all  these  three properties are localised (to the closed unit ball) versions of the three well known geometric properties namely, the  Radon Nikodym Property ($RNP$), Point of Continuity Property ($PCP$) and  Strong regularity ($SR$ ), see \cite{GGMS} and \cite {BS}.

\begin{definition}
Let $(X,\|.\|)$ be a Banach space. Then norm $\|.\|$ of $X$ is said to be Fr$\acute{e}$chet differentiable at $x\in X$ if
\begin{equation}\notag
	\lim_{\|y\|\rightarrow 0} \frac{\|x+y\|+\|x-y\|-2\|x\|}{\|y\|}=0
\end{equation}
\end{definition}

\begin{definition} \label{A}
For $\varepsilon>0$, the norm of $X$ is said to be
\begin{enumerate}
\item \cite{JZ} $\varepsilon$-rough if diameter of every $w^*$ slice of $B_{X^*}$ is greater than or equal to $\varepsilon.$ 
The norm of $X$ is said to be rough if it is $\varepsilon$-rough for some $\varepsilon>0$ and it is said to be non rough otherwise .
\item \cite{HLP} weakly $\varepsilon$-average rough if
every $w^*$ open subset in  $B_{X^*}$ has diameter greater than or equal to $\varepsilon.$ 
\item \cite{D} $\varepsilon$-average rough if every convex combination of  $w^*$ slice of $B_{X^*}$ has diameter greater than or equal to $\varepsilon.$
\end{enumerate}
\end{definition}
Rough norms were first studied by Day \cite{Da}, Kurzweil \cite{Ku} and Leach and Whitfield \cite{LW}. Subsequently, it was characterized in terms of diameter of slices of closed unit ball by John and Zizler in \cite{JZ}. The authors also 
introduced the notion of $\varepsilon$-rough norms and gave an example of a  Banach space having a non rough norm with no
point of Gateaux differentiability  (see \cite{P}). Later,  Deville \cite{D} introduced the notion of $\varepsilon$-average rough norms and characterized these norms in terms of diameter of weakly open subsets of the closed unit ball. He also  showed that if $X$ is a separable Banach space then $X$ has an equivalent $\varepsilon$-average rough norm for some $\varepsilon >0$ if and only if $X$ has an equivalent 2-average rough norm. Subsequently, Godefroy
showed in \cite{Go}, that a Banach space $X$ is Asplund  ( see \cite{B1}) if and only if every equivalent norm on $X$ is non-rough. From $\check{S}$mulyan Lemma \cite{DGZ} it is known that
if the norm of $X$ is Fr$\acute{e}$chet differentiable at least at one point on $S_X$, then the norm of $X$ is non rough. But the converse is not true. There is an equivalent non rough norm on $l_1(\Gamma)$ ($\Gamma$ uncountable) with no point of Fr$\acute{e}$chet differentiability \cite{JZ}.
Haller, Langemets and  P$\tilde{o}$ldvere \cite{HLP} introduced the notion of weakly $\varepsilon$-average rough norms. The notion of   $2$-rough norms and its variations  are well studied in literature. It is known that a Banach space $X$ can be equivalently renormed to have $2$-average rough norm if and only if $X$ contains an isomorphic copy  of $l_1$ (see \cite{DGZ}). 
There is a duality between the notion  of  $2$-rough norms  
and the diameter $2$ properties.  For more details, see   \cite{HLP1}, \cite{BGLPRZ}, \cite{BGLPRZ1}.

\begin{definition}
Let $Y\subset X$ be a subspace of $X.$ The annihilator of $Y$ in the dual space $X^*$ is the subspace of $X^*$ defined by $Y^\perp=\{x^*\in X^*: x^*(y)=0\ \forall y\in Y\}.$
\end{definition}

\begin{definition}
	Let $X$ be a Banach space, $Y$ be a closed subspace of $X$ and $F$ be a closed subspace of $X^*$. Then
	\begin{enumerate}
		\item $Y$ is said to be an ideal in $X$ if $Y^{\perp}$ is the kernel of a norm one projection in $X^*.$
		\item $Y$ is said to be  an $M$-ideal if there exists a linear projection $P$ : $X^* \rightarrow X^*$ with Ker $P$ = $Y^\perp$ and $\Vert x \Vert = \Vert Px \Vert + \Vert x-Px \Vert$ for all $x \in X.$
		\item $F$ is said to be norming subspace for $X$ if $\|x\| =
		\sup\{|x^*(x)| : x^* \in B_F\}$, for all $x \in X$.
		\item $Y$ is said to be  a strict ideal if for a projection $P: X^\ast \rightarrow X^\ast$
		with $\|P\| = 1$, $ker(P) = Y^\bot$, one also has  $B_{P(X^\ast)}$ is $w^*$-dense in $B_{X^*}$ or
		in other words $B_{P(X^*)}$ is a norming set for $X.$ 
	\end{enumerate}
\end{definition}

Let $Y$ be a subspace of $X.$ If $Y$ and $X/Y$ have a property (P) implies $X$ has same the property (P), then (P) is called a three space property.

In this work,  we introduce two weaker notions of non rough norms namely, the weakly average non rough norm and the average non rough norm. We  then establish a duality between the three versions of the non rough norms and the three versions  of  the small diameter properties.
We also discuss stability property of the non rough norms, with respect to   $l_p$  ( $1\leq p \leq \infty$) sum and ideals of Banach spaces.
We further  study the three space property in the context to  the  non rough norms. We prove that if $Y$ is a finite dimensional subspace of $X$ and $P: X\rightarrow Y$ is a norm one, continuous, onto linear projection, then the norm of $X/Y$ is non rough implies that the norm of  $X$ is non rough. We further show that if $Y$ is a finite dimensional (reflexive) subspace of $X$ such that the norm of $X/Y$ is weakly average non rough (average non rough), then the norm of $X$ is weakly average non rough (average non rough). 
 In some of our proofs, we use similar techniques as in \cite{BGLPRZ}.

\section{Nonrough norms}
We begin with two definitions:
\begin{definition} \label{def non}
The norm of $X$ is
\begin{enumerate}
\item weakly average non rough if the norm of $X$ is not weakly $\varepsilon$-average rough for any $\varepsilon>0.$ 
\item average non rough if the norm of $X$ is not $\varepsilon$-average rough for any $\varepsilon>0.$ 
\end{enumerate} 
\end{definition}

The following proposition is immediate.

\begin{proposition} \label{dual 1}
The norm of $X$ is non rough (resp. weakly average non rough, average non rough) if and only if $X^*$ has $w^*$-$BDP$ (resp. $w^*$-$BHP$, $w^*$-$BSCSP$).
\end{proposition}
Let us recall the following result from \cite{BS}.
\begin{proposition} \label{bs theorem}
\cite[Proposition 2.1]{BS}
A Banach space X has $BDP$ (resp. $BHP$ , $BSCSP$ ) if and only if $X^{**}$ has $w^*$-$BDP$ (resp. $w^*$-$BHP$, $w^*$-$BSCSP$).
\end{proposition}
The following proposition is immediate from Proposition $\ref{dual 1}$ and Proposition $\ref{bs theorem}$.

\begin{proposition} \label{dual 1 coro}
  The norm of  $X^*$ is non rough (resp. weakly average non rough, average non rough) if and only if $X$ has $BDP$ (resp. $BHP$, $BSCSP$).
\end{proposition}

Using Proposition 2.16 and Proposition 2.17 \cite {BS} and
Proposition $\ref{dual 1},$ , we have the following result. We omit the easy proof.
\begin{proposition} \label{w star}
Let $X$ and $Y$ be Banach spaces. Then the following hold. 
\begin{enumerate}
\item $X\oplus_1 Y$ has weakly average non rough (non rough) norm  if and only if  $X$ and $Y$ has weakly average non rough (non rough) norm.
\item $X\oplus_p Y$ ($1< p \leqslant \infty$) has weakly average non rough (non rough) norm  if and only if $X$ or $Y$ has weakly average non rough (non rough) norm.
\item If $X\oplus_1 Y$ has  average non rough norm , then both $X$ and $Y$ has average non rough norm.
\item  If $X$ or $Y$ has average non rough norm , then $X\oplus_p Y$ ($1< p \leqslant \infty$) has average non rough norm.
\item $X\oplus_{\infty} Y$ has average non rough norm  if and only if   $X$ or $Y$ has average non rough norm.
\end{enumerate}

\end{proposition}

We recall that for a compact Hausdorff space
$K$, $C(K,X)$ denotes the space of continuous $X$-valued functions
on $K$, equipped with the supremum norm. We recall from \cite{L1} that dispersed compact Hausdorff spaces have isolated points.

From Propositions 2.8, 2.10, Remarks 2.7, 2.9, 2.11 \cite {BS2} (Also see \cite {B2}  and \cite{BR} )  and  Proposition $\ref{dual 1},$ we have the following. We omit the easy proof. 
\begin{proposition}
Let $X$ be a Banach space and $Y$ be an ideal in $X.$
\begin{enumerate}
\item If $Y$ is an $M$-ideal in $X$ and $Y$ has non rough (respectively, weakly average non rough,  average non rough) norm, then  $X$ has non rough (respectively, weakly average non rough, average non rough) norm.
\item If $Y$ is a strict ideal in $X$ and $Y$ has non rough (respectively, weakly average non rough, average non rough) norm, then $X$ has non rough (respectively, weakly average non rough, average non rough) norm.
\item Let $K$ be a compact Hausdorff space with an isolated point. 
If $X$ has non rough (respectively, weakly average non rough, average non rough) norm, then $C(K,X)$ has non rough (respectively, weakly average non rough, average non rough) norm.
\end{enumerate}
\end{proposition}

We now explore the  non rough norm in the context of three space property. For that, we need the following results.

\begin{theorem}\label{dobt}
\cite[Theorem 5.3]{SSW}
Let $P:X\rightarrow X$ be 
a norm one, continuous
 linear projection with Ker(P) finite dimensional. Then
$\inf \{$diam(S): S is a slice of $B_X\} \leqslant \inf \{$diam(T): T is a slice of $B_{P(X)}\}.$
\end{theorem}

Proceeding similarly as in \cite[Theorem 5.3]{SSW}, we have the following :

\begin{theorem}\label{dobt1}
Let $P:X^*\rightarrow X^*$ be
a norm one, continuous linear projection with Ker(P) finite dimensional. Then
$\inf \{$diam(S): S is a $w^*$-slice of $B_{X^*}\} \leqslant \inf \{$diam(T): T is a $w^*$-slice of $B_{P(X^*)}\}.$
\end{theorem}

The following corollary is immediate. 

\begin{corollary}\label{prop weak star bdp 3 space}
Let $Y$ be a closed subspace of $X$ and $P:X^*\rightarrow Y^{\perp}$ be a norm one, continuous, onto linear projection with Ker(P) finite dimensional. If $Y^{\perp}$ has $w^*$-$BDP,$ then $X^*$ has $w^*$-$BDP.$ 
\end{corollary}

Since any finite dimensional space and any reflexive space has $BDP$ and hence all other small diameter properties,  we assume that in our following discussion without mentioning it separately in each case.

\begin{theorem}
Let $Y$ be a finite dimensional subspace of $X$ such that there exists a norm one continuous linear projection $P:X\rightarrow Y.$ If the norm of $X/Y$ is non rough, then  the norm of $X$ is non rough.
\end{theorem}

\begin{proof}
Since $Y$ is finite dimensional,  $Y^*$ is also finite dimensional.  Hence Ker$(P)^{\perp}(=$Image$(P^*))$ is also finite dimensional. Consider, 
$Q= I-P^* i^*$ where $i:Y\rightarrow X$ is the inclusion map and $I$ is the identity map on $X^*.$ Then Ker$(Q)$=Ker$(P)^{\perp}$ and Image$(Q)=Y^{\perp}.$ Thus, $\frac{Q}{\|Q\|} : X^* \rightarrow Y^{\perp}$ is a norm one, continuous, onto linear projection.
 Also, taking into account the duality $(X/Y)^*= Y^{\perp}$ and Proposition $\ref{dual 1},$ we have, $Y^{\perp}$ has $w^*$-$BDP.$ Finally, applying 
Corollary $\ref{prop weak star bdp 3 space}$ and Proposition $\ref{dual 1}$ successively, we get our desired result.
\end{proof}

\begin{proposition} \label{prop weak star bhp 3 space}
Let $Y$ be a subspace of $X$ such that $X^*/Y^{\perp}$ is finite dimensional and 
$Y^{\perp}$ has $w^*$-$BHP.$ Then  $X^*$ has $w^*$-$BHP.$
\end{proposition}

\begin{proof}

We prove by contradiction. Suppose $X^*$ does not have  $w^*$-$BHP.$ Then there exists $\varepsilon>0$ such that diameter of any nonempty relatively $w^*$ open set in $B_{X^*}$ is greater than $\varepsilon.$ 
By using the fact $Y^{\perp}$ has $w^*BHP,$  we get a nonempty relatively $w^*$ open set
 $$W=\{y^*\in B_{Y^{\perp}} :|(y^*-y_0^*)(y_i)|<\varepsilon_i \quad \forall i=1,2,\ldots,n\}$$ 
 (where $y_0^*\in B_{Y^{\perp}}$ and $y_1,\ldots,y_n \in X$)
 in $B_{Y^{\perp}}$ with diameter less than $\frac{\varepsilon}{2}.$ 
	
	For each $i=1,2,\ldots,n$, we choose 
	$\tilde{\varepsilon_i}>0$
	and $0<\delta_0<\frac{\varepsilon}{4}$ such that 
	$\tilde{\varepsilon_i} + \delta_0\|y_i\|<\varepsilon_i.$ 

	Define, 
	$$U=\{x^*\in B_{X^*}: |(x^*-y_0^*)(y_i)|<\tilde{\varepsilon_i} \quad \forall i=1,2,\ldots,n\}.$$
	Then $U\neq\emptyset$ and $U$ is relatively $w^*$ open in $B_{X^*}.$
	 Define $P:X^*\rightarrow X^*/{Y^{\perp}}$ by $P(x^*)=x^*+Y^{\perp}.$ 
	 Then $P$ is an onto, linear and bounded map.
	Thus, by the Open Mapping Theorem, $P$ is an open map. Also $y_0^*\in U\cap Y^{\perp}$  and hence
	  $P(U)$ is norm open set containing zero in $X^*/{Y^{\perp}}.$ 
	We choose $0<\delta<\frac{\delta_0}{2}$ such that $B(0,\delta)\subset P(U).$ 
	Consider, $B=P^{-1}(B(0,\delta))\bigcap U \subset B_{X^*}.$ 
	Taking into account norm-norm continuous implies $w^*-w^*$  continuous and  $X^*/{Y^{\perp}} $ is finite dimensional, we get  $B$ is a nonempty relatively  $w^*$ open set in $B_{X^*},$ and
	hence diameter of $B$ is greater than $\varepsilon.$
	 Therefore, there exist $v_1^*,v_2^* \in B$ such that $\|v_1^*-v_2^*\|>\varepsilon.$ Since $\|P(v_1^*)\| <\delta,$ then $d(v_1^*,Y^{\perp})<\delta$ and hence, there exists $u_1^*\in Y^{\perp}$ such that $\|u_1^*-v_1^*\|<\delta.$
	 Similarly for $v_2^*$ there exists $u_2^*\in Y^{\perp}$ such that $\|u_2^*-v_2^*\|<\delta.$ 
	 Without loss of generality we can assume that $u_1^*,u_2^*\in B_{Y^{\perp}},$ otherwise, we consider $\frac{u_1^*}{\|u_1^*\|}$ and $\frac{u_2^*}{\|u_2^*\|}.$ \\
	 For each $i=1,2,\ldots,n$ ,
	  \begin{equation}\notag
	|(u_1^*-y_0^*)(y_i)|\leqslant |(u_1^*-v_1^*)(y_i)|+|(v_1^*-y_0^*)(y_i)|\leqslant \|y_i\| 2 \delta + \tilde{\varepsilon_i} <\| y_i\| \delta_0 + \tilde{\varepsilon_i}<\varepsilon_i 
	 \end{equation}
	 
Similar is also true for $u_2^*$. Thus, $u_1^*, u_2^* \in W.$\\
Now, 	\begin{equation}\notag \|u_1^*-u_2^*\|\geqslant \|v_1^*-v_2^*\|+\|u_1^*-v_1^*\|+\|u_2^*-v_2^*\| >\varepsilon-4\delta>
	\varepsilon-2\delta_0>\varepsilon-\frac{\varepsilon}{2}=\frac{\varepsilon}{2}	
	\end{equation}
	 
	Hence, $W$ has diameter greater than $\frac{\varepsilon}{2},$  a contradiction.
\end{proof}

\begin{theorem}\label{M1}
	Let $Y$ be a finite dimensional subspace of $X$ such that norm of $X/Y$ is weakly average non rough, then norm of $X$ is weakly average non rough.
\end{theorem}

\begin{proof}
Taking into account the duality $X^*/Y^{\perp} = Y^*$ and $Y$ is finite dimensional, we get $X^*/Y^{\perp}$ is finite dimensional. Again from the duality $(X/Y)^*= Y^{\perp}$ and Proposition $\ref{dual 1},$ we have $Y^{\perp}$ has $w^*$-$BHP.$ Finally, applying 
Proposition $\ref{prop weak star bhp 3 space}$ and Proposition $\ref{dual 1}$ successively, we get our desired result.

\end{proof}

\newpage

We recall the following theorem : 
 \begin{theorem} \label{Phelp's lemma}
\cite[Theorem 5.21]{P} A Banach space $X$ has $RNP$ if and only if every bounded closed convex subset of $X$ is the closed convex hull of its strongly exposed points.
\end{theorem}

\begin{proposition} \label{prop weak star bscsp 3 space}
Let $Y$ be a subspace of $X$ such that $X^*/Y^{\perp}$ is reflexive and 
$Y^{\perp}$ has $w^*$-$BSCSP.$ Then  $X^*$ has $w^*$-$BSCSP.$
\end{proposition}

\begin{proof}
We prove by contradiction. Suppose  $X^*$ does not have  $w^*$-$BSCSP.$
	 Then there exists $\varepsilon> 0$ 
	 such that diameter of any  convex combination of $w^*$ slices of $B_{X^*}$ is greater than $\varepsilon.$
	  Since $Y^{\perp}$ 
	  has $w^*$-$BSCSP,$ then there exists convex combination of slices $\sum_{i-1}^{n}  \lambda_i S(B_{Y^{\perp}},y_i,\delta)$ in $B_{Y^{\perp}}$ with diameter less than $\frac{\varepsilon}{2}.$ 
Choose $\tilde{\varepsilon},\delta_0> 0$ such that $\tilde{\varepsilon}+2 \delta_0< \delta$ and $ 0<\delta_0<\frac{\varepsilon}{8}.$ 
Define $P:X^*\rightarrow X^*/{Y^{\perp}}$ by $P(x^*)=x^*+Y^{\perp}.$ 
	 Observe, $A_i=P(S(B_{X^*},y_i,\tilde{\varepsilon}))$  is a convex subset of $B_{X^*/{Y^{\perp}}}$ containing zero.
Since $X^*/{Y^{\perp}}$ is reflexive, it has $RNP$  ( see \cite{DU}). Thus
	 by  Theorem $\ref{Phelp's lemma}$, for each $i=1,2,\ldots,n,$
	 $\bar{A}_i= 
	 \ov{co}($strongly exposed points  in $\bar{A}_i)$
	   and hence there exists a convex combination of strongly exposed points
	   $a_i^*=\sum_{j=1}^{n_i} \gamma_j^i (a_j^i)^*$ of $\bar{A}_i$  such that 
	   $\|a_i^*\|<\delta_0.$ 

For each i and j,	let  $(a_j^i)^*$ is strongly exposed by $a_j^i \in S_{({X^*/{Y^{\perp}}})^*} (= S_{Y^{**}}=S_Y)$ 

and choose $\eta_j^i>0$ such that diam$((S(B_{X^*/{Y^{\perp}}}, a_j^i,\eta_j^i)\cap\bar{A}_i)<\delta_0.$ 
	   	   
	Since  $S(B_{X^*/{Y^{\perp}}}, a_j^i,\eta_j^i)\cap\bar{A}_i \neq \emptyset,$ then $S(B_{X^*/{Y^{\perp}}}, a_j^i,\eta_j^i)\cap A_i \neq \emptyset$ which gives

	$S(B_{X^*}, P^*a_j^i,\eta_j^i)\cap S(B_{X^*},y_i,\tilde{\varepsilon}) \neq \emptyset.$  Indeed, let $P z^*\in S(B_{X^*/{Y^{\perp}}}, a_j^i,\eta_j^i)$ for some $z^* \in S(B_{X^*},y_i,\tilde{\varepsilon}).$ Hence,
$ (P^* a_j^i)(z^*)=(P z^*)(a_j^i)>1 - \eta_j^i >  \sup_{w^*\in B_{X^*}} (P^* a_j^i) (w^*) - \eta_j^i.$

	 Now, $ D =\sum_{i=1}^{n} \lambda_i \sum_{j=1}^{n_i}\gamma_j^i ((S(B_{X^*},P^* a_j^i,\eta_j^i)\cap S(B_{X^*},y_i,\tilde{\varepsilon}))$ is a convex combination of nonempty relatively $w^*$ open subset of $B_{X^*}.$
	 From Bourgain's lemma, $D$ contains a convex combination of $w^*$ slices of $B_{X^*}$, see \cite{B3}. By our assumption,  diameter of any convex combination of $w^*$ slices of $B_{X^*}$ is greater than $\varepsilon.$ Hence 
	  diam$(D)> \epsilon.$
	 Thus for each i and j there exist $(x_j^i)^*,(z_j^i)^* \in S(B_{X^*},P^* a_j^i,\eta_j^i)\cap S(B_{X^*},y_i,\tilde{\varepsilon})$ such that $$\|\sum_{i=1}^{n} \lambda_i \sum_{j=1}^{n_i}\gamma_j^i (x_j^i)^*-\sum_{i=1}^{n} \lambda_i \sum_{j=1}^{n_i}\gamma_j^i   (z_j^i)^* \|>\varepsilon.$$
	   Observe, $$\sum_{j=1}^{n_i}\gamma_j^i (x_j^i)^* \in   \sum_{j=1}^{n_i}\gamma_j^i \Big(S(B_{X^*},P^* a_j^i,\eta_j^i)\cap S(B_{X^*},y_i,\tilde{\varepsilon})\Big)$$ 
	    $$ \Rightarrow  P( \sum_{j=1}^{n_i}\gamma_j^i (x_j^i)^*)\in \sum_{j=1}^{n_i}\gamma_j^i (S(B_{X^*/{Y^{\perp}}}, a_j^i,\eta_j^i)\cap A_i ).$$ 
Thus, for each $i=1,2,\ldots,n,$ diam $(\sum_{j=1}^{n_i}\gamma_j^i (S(B_{X^*/{Y^{\perp}}}, a_j^i,\eta_j^i)\cap A_i ) <\delta_0$ gives,	
$$\|P( \sum_{j=1}^{n_i}\gamma_j^i (x_j^i)^*) - \sum_{j=1}^{n_i}\gamma_j^i (a_j^i)^*\|<\delta_0$$
Thus $\|P( \sum_{j=1}^{n_i}\gamma_j^i (x_j^i)^*)\|<\delta_0+\| \sum_{j=1}^{n_i}\gamma_j^i (a_j^i)^*\| <2\delta_0$ and hence $d(\sum_{j=1}^{n_i}\gamma_j^i (x_j^i)^*,Y^{\perp})<2\delta_0$.
   

 Thus, there exists $v_i^*\in B_{Y^{\perp}}$ such  that  $\|v_i^* - \sum_{j=1}^{n_i}\gamma_j^i (x_j^i)^*\|<2\delta_0.$  

	Similarly, for each $i=1,2,\ldots,n$ there exists $w_i^*\in B_{Y^{\perp}}$ such that $\|w_i^* - \sum_{j=1}^{n_i}\gamma_j^i (z_j^i)^*\|<2\delta_0.$ \\
	For each $i=1,2,\ldots,n$, $$v_i^*(y_i)=\sum_{j=1}^{n_i}\gamma_j^i (x_j^i)^*(y_i)+(v_i^*-\sum_{j=1}^{n_i}\gamma_j^i (x_j^i)^*)(y_i)>1-\tilde{\varepsilon}-2\delta_0
$$
	Similarly, $ w_i^*(y_i)>1-\tilde{\varepsilon}-2\delta_0$ and hence $v_i^*,w_i^*
	\in S(B_{Y^{\perp}},y_i,\tilde{\varepsilon}+2\delta_0)$\\
	Finally, 
	\begin{equation}\notag
	\begin{split}
	 \|\sum_{i=1}^{n} \lambda_i v_i^* - \sum_{i=1}^{n} \lambda_i w_i^*\| 
	\geqslant \|\sum_{i=1}^{n} \lambda_i \sum_{j=1}^{n_i}\gamma_j^i (x_j^i)^* - \sum_{i=1}^{n} \lambda_i \sum_{j=1}^{n_i}\gamma_j^i (z_j^i)^*  \|\hspace{4 cm} \\ 
	-\|\sum_{i=1}^{n} \lambda_i (v_i)^*-\sum_{i=1}^{n} \lambda_i \sum_{j=1}^{n_i}\gamma_j^i (x_j^i)^*  \| -\|\sum_{i=1}^{n} \lambda_i w_i^* - \sum_{i=1}^{n} \lambda_i \sum_{j=1}^{n_i}\gamma_j^i (z_j^i)^* \| \\ 
	> \varepsilon - \sum_{i=1}^{n} \lambda_i \|v_i^* - \sum_{j=1}^{n_i}\gamma_j^i (x_j^i)^*   \|-\sum_{i=1}^{n} \lambda_i \|w_i^*- \sum_{j=1}^{n_i}\gamma_j^i (z_j^i)^*   \| \hspace{1 cm}\\
	>\varepsilon -4\delta_0 \hspace{8 cm} 
	\end{split}
	\end{equation}	

Hence, $\mathrm{diam} (\sum_{i=1}^{n} \lambda_i  S(B_{Y^{\perp}},y_i,\tilde{\varepsilon}+2\delta_0) \geqslant \varepsilon-4\delta_0.$\\
	Since $\sum_{i=1}^{n} \lambda_i  S(B_{Y^{\perp}},y_i,\tilde{\varepsilon}+2\delta_0) \subset \sum_{i=1}^{n} \lambda_i  S(B_{Y^{\perp}},y_i,\delta),$ we have\\
$\mathrm{diam}( \sum_{i=1}^{n} \lambda_i  S(B_{Y^{\perp}},y_i,\delta) \geqslant \mathrm{diam} (\sum_{i=1}^{n} \lambda_i  S(B_{Y^{\perp}},y_i,\tilde{\varepsilon}+2\delta_0))>\varepsilon-4\delta_0 >\frac{\varepsilon}{2}$,  a contradiction.
\end{proof}

\begin{theorem}\label{M2}
	Let $Y$ be a reflexive subspace of $X$ such that the norm of $X/Y$ is  average non rough, then the norm of $X$ is average non rough.
\end{theorem}

\begin{proof}
Taking into account the duality $X^*/Y^{\perp} = Y^*$ and  reflexivity of $Y,$  we conclude that  $X^*/Y^{\perp}$ is reflexive. Again from the duality $(X/Y)^*= Y^{\perp}$ and Proposition $\ref{dual 1},$ we have $Y^{\perp}$ has $w^*$-$BSCSP.$ Finally applying 
Proposition $\ref{prop weak star bscsp 3 space}$ and Proposition $\ref{dual 1}$ successively, we get our desired result.
 
\end{proof}

\begin{Acknowledgment}
The authors thank Professor Abraham Rueda Zoca for valuable suggestions. The first author's research is supported by the dean's supplemental fund provided by the College of Arts and Sciences, Loyola University, Maryland, USA.The second  author's research is funded by the National Board for Higher Mathematics (NBHM), Department of Atomic Energy (DAE), Government of India, Ref No: 0203/11/2019-R$\&$D-II/9249.
\end{Acknowledgment}

\end{document}